\documentclass{amsart}
\usepackage{amsmath,amsfonts,amsthm}
\usepackage{amssymb,latexsym}
\usepackage{graphics}
\usepackage[colorlinks]{hyperref}

\usepackage[all]{xypic}

\theoremstyle{plain}
\theoremstyle{definition}
\newtheorem{theorem}{Theorem}[section]
\newtheorem{lemma}[theorem]{Lemma}

\newtheorem{example}[theorem]{Example}

\newtheorem{problem}[theorem]{Problem}

\newtheorem{note}[theorem]{Note}
\newtheorem{jadval}[theorem]{Table}

\theoremstyle{remark}

\numberwithin{equation}{section}
\newcommand{\SP}{\: \: \: \: \:}

\title[Set--theoretical entropies of Euler's totient function]{Set--theoretical entropies of Euler's totient function and other number theoretical special functions}
\author[F. Ayatollah Zadeh Shirazi, R. Yaghmaeian]{Fatemah Ayatollah Zadeh Shirazi, Reza Yaghmaeian}
\begin{document}
\begin{abstract}
In the following text we show set--theoretical entropy of
Euler's totient function and contravariant 
set--theoretical entropy of
Dedekind psi function are zero. Also 
contravariant set--theoretical entropy of
Euler's totient function and 
set--theoretical entropy of
Dedekind psi function are $+\infty$. We pay attention to some
of the other
number theoretical special functions too. We continue
our studies on Alexandroff topologies induced by 
Euler's totient function and Dedekind psi function.
\end{abstract}
\maketitle
\noindent {\small {\bf 2010 Mathematics Subject Classification:}  11A25, 11Y70, 33F99  \\
{\bf Keywords:}} Alexandroff topology, 
Dedekind psi function, Euler's totient function, 
Infinite orbit number, Infinite anti--orbit number, Set--theoretical entropy.
\section{Introduction}
\noindent Various types of entropies have been studied in different 
branches of mathematics. In category Set one may consider
set--theoretical and contravariant set--theoretical entropies
of self--maps. Our main aim in this text is to study 
set--theoretical behaviour 
of some well--known number theoretical maps like 
Euler's totient function, Dedekind psi function and
their generalizations. However 
set--theoretical and contravariant set--theoretical entropies of
a ``{\it nice}'' self--maps have interactions with 
infinite orbit number's concept and infinite anti--orbit
number's concept so we pay attention to these concepts too.
We continue our studies in topological
arising concepts in this regard, our main emphasis in topological
point of view deals with Alexandroff topological spaces' approach. 
\\
Let ${\mathbb N}=\{1,2,\ldots\}$ be the set of natural numbers
and ${\mathbb P}=\{2,3,5,7,\ldots\}$ the set of prime numbers.
For finite set $A$ by
$\sharp A$ we mean the number of elements of $A$. Also we say 
$\lambda:X\to X$ is finite fibre if $\lambda^{-1}(x)$
is finite for all $x\in X$.
\\
{\bf Background on infinite orbit number and infinite anti--orbit
number of a self--map.}
For self--map $\lambda:X\to X$ we say
the one--to--one sequence $\{a_n\}_{n\geq1}$ is:
\\
$\bullet$ an infinite $\lambda-$orbit if for all $n\geq1$, $a_{n+1}=\lambda(a_n)$,
\\
$\bullet$ an infinite $\lambda-$anti--orbit if for all $n\geq1$, $a_n=\lambda(a_{n+1})$ (see e.g., \cite[Definition 1.2]{DGV} and \cite[Definition 1.1]{GV}).
\\
Moreover we set \cite{AKS}:
\begin{center}
${\mathfrak o}(\lambda):=\sup(\{n\geq1:$ there exists $n$ disjoint infinite $\lambda-$orbit sequences$\}\cup\{0\})$,
\\
${\mathfrak a}(\lambda):=\sup(\{n\geq1:$ there exists $n$ disjoint infinite $\lambda-$anti--orbit sequences$\}\cup\{0\})$,
\end{center}
we call ${\mathfrak o}(\lambda)$ infinite orbit number of $\lambda$ and
${\mathfrak a}(\lambda)$ infinite anti--orbit number of $\lambda$.
\newpage
\noindent {\bf Background on set--theoretical and contravariant set--theoretical entropies.}
For $\lambda:X\to X$ and finite subset $A$ of $X$ the following
limit exists
\[{\rm ent}_{set}(\lambda,A)=\mathop{\lim}\limits_{n\to\infty}\dfrac{\sharp(A\cup\lambda(A)\cup\cdots\cup\lambda^{n-1}(A))}{n}\]
and we call ${\rm ent}_{set}(\lambda):=\sup\{{\rm ent}_{set}(\lambda,B):B$ is a finite subset of $X\}$ set--theoretical entropy of $\lambda$ \cite{AD}. 
Moreover for finite fibre onto map $\mu:X\to X$ and
finite sunset $A$ of $X$ the following
limit exists
\[{\rm ent}_{cset}(\mu,A)=\mathop{\lim}\limits_{n\to\infty}\dfrac{\sharp(A\cup\mu^{-1}(A)\cup\cdots\cup\mu^{-(n-1)}(A))}{n}\]
and we call ${\rm ent}_{cset}(\mu):=\sup\{{\rm ent}_{cset}(\mu,B):B$ is a finite subset of $X\}$ set--theoretical entropy of $\mu$.
On the other hand if $\lambda:X\to X$ is finite fibre and
$sc(\lambda):=\mathop{\bigcap}\limits_{n\geq1}\lambda^n(X)$,
then $\lambda\restriction_{sc(\lambda)}:sc(\lambda)\to sc(\lambda)$
is finite fibre and onto, we call 
${\rm ent}_{cset}(\lambda):={\rm ent}_{cset}
(\lambda\restriction_{sc(\lambda)})$ contravariant
set--theoretical entropy of $\lambda$ \cite{DG}. 
\begin{note}
For $\lambda:X\to X$ we have ${\rm ent}_{set}(\lambda)={\mathfrak o}(\lambda)$ \cite[Proposition 2.16]{AD}, also for finite fibre
$\lambda:X\to X$ we have ${\rm ent}_{cset}(\lambda)={\mathfrak a}(\lambda)$ \cite[Theorems 3.2, 3.9]{DG}
\end{note}
\noindent {\bf Some number theoretical special functions.}
Let's recall the following functions ($n\geq1$ and for convenient suppose all of them map
$1$ to $1$):
\\
$\bullet$ Jordan's totient function (for $k\geq1$): $J_k(n)=\{(s_1,\ldots,s_k):s_1,\ldots,s_k\in\{1,\ldots,n\},\gcd(s_1,\ldots,s_k,n)=1\}(n^k\prod\{1-\frac1{p^k}:p\in{\mathbb P},p|n\})$ 
(so well--known Euler's totient function $\varphi$ is $J_1$) 
(see e.g., \cite{TV})
\\
$\bullet$ Generalized Dedekind psi function (for $k\geq1$): $\psi_k(n)=n^k\prod\{1+\frac1{p^k}:p\in{\mathbb P},p|n\}=\frac{J_{2k}(n)}{J_k(n)}$ (so well--known Dedekind psi function, $\psi(=\psi_1)$ is $\frac{J_2}{J_1}$) \cite{S}
\\
$\bullet$ Unitary totient function: $\varphi^*(n)=\prod\{p^\alpha-1:p^\alpha|n,p^{\alpha+1}\not|n,p\in{\mathbb P},\alpha\geq1\}$ (see e.g., \cite{L, SA})
\\
$\bullet$ $\Omega(n)=\sum\{\alpha:p^\alpha|n,p^{\alpha+1}\not|n,p\in{\mathbb P}\}$ \cite{N}
\\
$\bullet$ $\omega(n)=\sum\{1:p^\alpha|n,p^{\alpha+1}\not|n,p\in{\mathbb P},\alpha\geq1\}$
\\
$\bullet$ $d_l(n)=\sharp\{(s_1,\cdots,s_l)\in{\mathbb N}^l:s_1\cdots s_l=n\}$ ($l\geq2$) (we denote $d_2$ with $d$) 
\cite{N}
\\
$\bullet$ $\sigma_l(n)=\mathop{\Sigma}\limits_{d|n,d\leq n}d^l$ ($l\geq2$) \cite{N}
\section{Infinite orbit number and infinite anti--orbit number of $\varphi$}
\noindent In this section we compute 
infinite orbit number and infinite anti--orbit number of
Euler's totiont function, Dedekind psi function and some other
well--known maps.
\begin{lemma}\label{taha70}
For $f:{\mathbb N}\to{\mathbb N}$ with $f(n)\leq n$ we have ${\mathfrak o}(f)=0$.
\end{lemma}
\begin{proof}
Suppose $\{x_n\}_{n\geq1}$ is an infinite $f-$orbit
and for all $n\geq1$ we have $f(n)\leq n$, thus 
$x_1,x_2=f(x_1),x_3=f^2(x_1),\ldots\in\{1,2,\ldots,x_1\}$, thus $\{x_n\}_{n\geq1}$ is not infinite and one--to--one sequence.
\end{proof}
\begin{lemma}\label{taha50}
For $f:{\mathbb N}\to{\mathbb N}$ with $f(n)\geq n$ we have ${\mathfrak a}(f)=0$.
\end{lemma}
\begin{proof}
Suppose $\{x_n\}_{n\geq1}$ is an infinite $f-$anti--orbit
and for all $n\geq1$ we have $f(n)\geq n$, thus 
for all $n\geq1$ we have $x_1=f^n(x_{n+1})\geq f^{n-1}(x_{n+1})\geq\cdots\geq x_{n+1}$, so
$x_1,x_2,x_3,\ldots\in\{1,2,\ldots,x_1\}$
thus $\{x_m\}_{m\geq1}$ is not infinite and one--to--one.
\end{proof}
\begin{lemma}\label{taha60}
For $f:{\mathbb N}\to{\mathbb N}$ with $f(n)> n$ for
all $n>1$, we have ${\mathfrak o}(f)>0$.
\end{lemma}
\begin{proof}
Let $x\geq2$, then $\{f^n(x)\}_{n\geq1}$ is an infinite 
$f-$orbit, thus 
${\mathfrak o}(f)>0$.
\end{proof}
\begin{lemma}\label{taha10}
For $k\geq1$ let $S_k:=\{2^k3^n\}_{n\geq1}$, then $S_1,S_2,\ldots$ are disjoint infinite $\varphi-$anti--orbit sequences, so ${\mathfrak a}(\varphi)=+\infty$. In addition for all $n\geq1$, $\varphi(n)\leq n$ so ${\mathfrak o}(\varphi)=0$.
\end{lemma}
\begin{proof}
For $n,k,s,t\geq1$ we have $\varphi(2^k3^{n+1})=2^{k-1}(2-1)3^k(3-1)=2^k3^n$ moreover $2^k3^n=2^s3^t$ if and only if $k=s$ and $n=t$.
\end{proof}
\begin{lemma}\label{taha20}
For $p\in{\mathbb P}\setminus\{2\}$ let 
$S_p:=\{x_n^p\}_{n\geq1}$ with
\begin{itemize}
\item $x_1^p=p$,
\item $x_{n+1}^p=p^{x_n^p-1}$ ($n\geq1$),
\end{itemize}
then $S_3,S_5,S_7,S_{11},\ldots$ are disjoint infinite $d-$anti--orbit sequences, so ${\mathfrak a}(d)=+\infty$. In addition for all $n\geq1$, $d(n)\leq n$ so ${\mathfrak o}(d)=0$.
\end{lemma}
\begin{proof}
For $n,m\geq1$ and $p,q\in{\mathbb P}\setminus\{2\}$ we have $d(x_{n+1}^p)=d(p^{x_n^p-1})=x_n^p-1+1=x_n^p$ moreover if $x_n^p=x_m^q$ then the unique prime divisor of $x_n^p$ is $p$ and   the unique prime divisor of $x_m^q$ is $q$, so $p=q$ thus $x_n^p=x_m^p$  moreover for all $i\geq1$ we have 
$x_i^p<x_{i+1}^p$, so $x_n^p=x_m^p$ leads to $n=m$.
\end{proof}
\begin{lemma}\label{taha30}
For $p\in{\mathbb P}$ let 
$S_p:=\{x_n^p\}_{n\geq1}$ with
\begin{itemize}
\item $x_1^p=p$,
\item $x_{n+1}^p=p^{x_n^p}$ ($n\geq1$),
\end{itemize}
then $S_2,S_3,S_5,S_7,S_{11},\ldots$ are disjoint infinite $\Omega-$anti--orbit sequences, so ${\mathfrak a}(\Omega)=+\infty$. In addition for all $n\geq1$, $
\Omega(n)\leq n$ so ${\mathfrak o}(\Omega)=0$.
\end{lemma}
\begin{proof}
For $n,m\geq1$ and $p,q\in{\mathbb P}$ we have $\Omega(x_{n+1}^p)=\Omega(p^{x_n^p})=x_n^p$ moreover if $x_n^p=x_m^q$ then the unique prime divisor of $x_n^p$ is $p$ and   the unique prime divisor of $x_m^q$ is $q$, so $p=q$ thus $x_n^p=x_m^p$  moreover for all $i\geq1$ we have 
$x_i^p<x_{i+1}^p$, so $x_n^p=x_m^p$ leads to $n=m$.
\end{proof}
\begin{lemma}\label{taha40}
For $p\in{\mathbb P}\setminus\{2\}$ let 
$S_p:=\{x_n^p\}_{n\geq1}$ with (suppose $q_n$ is the $n$th prime number):
\begin{itemize}
\item $x_1^p=p=q_j$,
\item $x_{n+1}^p=pq_{j+1}q_2\cdots q_{j+x_n^p-1}$ ($n\geq1$),
\end{itemize}
then $S_3,S_5,S_7,S_{11},\ldots$ are disjoint infinite $\omega-$anti--orbit sequences, so ${\mathfrak a}(\omega)=+\infty$. In addition for all $n\geq1$, $
\omega(n)\leq n$ so ${\mathfrak o}(\omega)=0$.
\end{lemma}
\begin{proof}
For $n,m\geq1$ and $p,q\in{\mathbb P}\setminus\{2\}$ we have 
\[\omega(x_{n+1}^p)=\omega(pq_{j+1}q_{j+2}\cdots q_{j+x_n^p-1})=x_n^p\]
moreover if $x_n^p=x_m^q$ then the least prime divisor of $x_n^p$ is $p$ and   the least prime divisor of $x_m^q$ is $q$, so $p=q$ thus $x_n^p=x_m^p$  moreover for all $i\geq1$ we have 
$x_i^p<x_{i+1}^p$, so $x_n^p=x_m^p$ leads to $n=m$.
\end{proof}
\begin{lemma}\label{salam10}
For $k\geq1$ let 
$S_k:=\{3^k2^n\}_{n\geq1}$,
then $S_1,S_2,\ldots$ are disjoint infinite $\psi-$orbit sequences, so ${\mathfrak o}(\psi)=+\infty$. In addition for all $n\geq1$, $
\psi(n)\geq n$ so ${\mathfrak a}(\psi)=0$.
\end{lemma}
\begin{proof}
For $n,m,s,t\geq1$ we have $\psi(3^m2^n)=3^{m-1}(3+1)2^{n-1}(2+1)=3^m2^{n+1}$ moreover if $3^m2^n=3^s2^t$ if and only if $m=s$ and $n=t$.
\end{proof}
\begin{lemma}\label{salam20}
For $k\geq1$ let 
$S_k:=\{2^{2^{n+1}k+2^n-1}3\}_{n\geq1}$,
then $S_1,S_2,\ldots$ are disjoint infinite $J_2-$orbit sequences, so ${\mathfrak o}(J_2)=+\infty$. In addition for all $n\geq1$, $
J_2(n)\geq n$ so ${\mathfrak a}(J_2)=0$.
\end{lemma}
\begin{proof}
For $n,m,s,t\geq1$ we have $J_2(2^{2^{n+1}m+2^n-1}3)=2^{2(2^{n+1}m+2^n-1)-2}(2^2-1)3^{2-2}(3^2-1)=2^{2(2^{n+1}m+2^n-1)+1}3=2^{2^{(n+1)}m+2^{n+1}-1}3$ moreover:
\begin{eqnarray*}
2^{2^{n+1}m+2^n-1}3=2^{2^{s+1}t+2^s-1}3 & \Leftrightarrow & 
	2^{n+1}m+2^n-1=2^{s+1}t+2^s-1 \\
& \Leftrightarrow & 2^n(2m+1)=2^s(2t+1) \\
& \Leftrightarrow & n=s\wedge m=t\:.
\end{eqnarray*}
\end{proof}
\begin{note}\label{salam30}
Suppose $f:\mathbb{N}\to\mathbb{N}$ is a multiplicative function
$g:\mathbb{P}\times\mathbb{N}\to\mathbb{N}\cup\{0\}$
and $h:\mathbb{P}\to\mathbb{N}\cup\{0\}$ such that $f(1)=1$ and
$f(p^n)=p^{g(p,n)}h(p)(\geq1)$ for all $p\in\mathbb{P}$
and $n\geq1$. Also suppose there exist distinct 
$p,q\in\mathbb{P}$ and $u,v\geq1$ with
$p^u=h(q)$ and $q^v=h(p)$ define
$s,t:\mathbb{N}\to\mathbb{N}$ with $s(x)=g(p,x)+u$, $t(x)=g(q,x)+v$.
If there exists $(x_1,y_1),(x_2,y_2),\ldots\in\mathbb{N}\times\mathbb{N}$ 
such that
$\mathop{\mathbb{N}\times\mathbb{N}\to\mathbb{N}\times\mathbb{N}\SP\SP
}\limits_{(n,m)\mapsto(s^n(x_m),t^n(y_m))}$ is one--to--one, then
$\mathfrak{o}(f)=+\infty$ since for $S_n:=\{p^{s^i(x_n)}q^{t^i(y_n)}\}_{i\geq1}$,
the sequences $S_1,S_2,\ldots$ are disjoint infinite $f-$orbit sequences
(use the fact that 
\begin{eqnarray*}
f(p^{s^i(x_n)}q^{t^i(y_n)}) & = & 
	p^{g(p,s^i(x_n))}h(p)q^{g(q,t^i(y_n))}h(q) \\
& = & p^{g(p,s^i(x_n))+u}q^{g(q,t^i(y_n))+v}=p^{s^{i+1}(x_n)}q^{t^{i+1}(y_n)}\: ).
\end{eqnarray*}
Lemmas~\ref{salam10} and~\ref{salam20} are examples of the above
construction.
\end{note}
\begin{note}
As a generalization of Note~\ref{salam30}
suppose $f:\mathbb{N}\to\mathbb{N}$ is a multiplicative function
$g:\mathbb{P}\times\mathbb{N}\to\mathbb{N}\cup\{0\}$
and $h:\mathbb{P}\to\mathbb{N}\cup\{0\}$ such that $f(1)=1$ and
$f(p^n)=p^{g(p,n)}h(p)(\geq1)$ for all $p\in\mathbb{P}$
and $n\geq1$. Also suppose there exist distinct 
$p_1,\ldots,p_m\in\mathbb{P}$ and $(u^i_1,\ldots,u^i_m)\in{\mathbb N}^m$
(for $i=1,\ldots,m$) with $u^i_j+g_(p_i,x)\geq1$ for all $i,j,x$ and
$p_1^{u_1^i}\cdots p_m^{u_m^i}=h(p_i)$ define
$s_i:\mathbb{N}\to\mathbb{N}$ with $s_i(x)=g(p_i,x)+
(u_i^1+\cdots+u_i^m)$.
If there exists $(x^1_1,x^2_1,\ldots,x_1^m),(x^1_2,x^2_2,\ldots,x_2^m),\ldots\in\mathbb{N}\times\mathbb{N}$ 
such that
$\mathop{\mathbb{N}^l\to\mathbb{N}^l\SP\SP\SP\SP\SP\SP\SP
}\limits_{(i,j)\mapsto(s_1^i(x^1_j),s_2^i(x^2_j),\ldots,s_m^i(x_j^m))}$ 
is one--to--one, then
$\mathfrak{o}(f)=+\infty$ since for $S_n:=\{p_1^{s_1^i(x^1_n)}p_2^{s_2^i(x^2_n)}\cdots p_m^{s_m^i(x^m_n)}\}_{i\geq1}$,
the sequences $S_1,S_2,\ldots$ are disjoint infinite $f-$orbit sequences.
\end{note}
\begin{jadval}\label{taha80}
We have the following table:
\[\begin{array}{l|l|c|c|}
& \lambda & {\mathfrak o}(\lambda) & {\mathfrak a}(\lambda) \\ \hline
1st. \: row & \varphi=J_1,d(=d_2),\Omega,\omega & 0 & +\infty \\ \hline
2nd. \: row & \varphi^* & 0 & \\ \hline
3rd. \: row & J_2, \psi(=\psi_1) & +\infty & 0 \\ \hline
4th. \: row & \sigma_k,\psi_k, J_{k+2} (k\geq1) & >0 & 0 \\ \hline
\end{array}\]
\end{jadval}
\begin{proof}
For the 1st. row use 
Lemmas~\ref{taha10},~\ref{taha20},~\ref{taha30},~and~\ref{taha40}.
\\
For the 2nd. row use 
Lemma~\ref{taha70}
\\
For the 3rd. row use 
Lemmas~\ref{salam10}~and~\ref{salam20}.
\\
For the 4th. row use Lemmas~\ref{taha50},~\ref{taha60} and  the
fact that for all $n\geq2$ we have $\sigma_k(n)>n,\psi_k(n)>n,J_{k+2}(n)>n$
(for $k\geq1$).
\end{proof}
\begin{note}\label{taha90}
{\bf 1.} For $f:{\mathbb N}\to{\mathbb N}$ with $f(n)\geq n$ (for
all $n\geq1$) we have $f^{-1}(m)\subseteq\{1,\ldots,m\}$ (for all
$m\geq1$) and $f:{\mathbb N}\to{\mathbb N}$ is finite fibre, thus
$\sigma_k,\psi_k,J_{k+1}$ (for $k\geq1$) are finite fibre.
\\
{\bf 2.} For distinct prime numbers $p_1,\ldots,p_n$
and $\alpha_1,\ldots,\alpha_n\geq1$ with
$\varphi(p_1^{\alpha_1}\cdots p_n^{\alpha_n})
=m$ we have
$p_i-1\leq m$ and $2^{\alpha_i-1}\leq p_i^{\alpha_i-1}\leq m$ for all $i=1,\ldots,n$,
so $p_1,\ldots,p_n\leq m+1$ and $\alpha_1,\ldots,\alpha_n\leq \dfrac{\log m}{\log 2}+1$ 
therefore 
\[p_1^{\alpha_1}\cdots p_n^{\alpha_n}\leq\prod\left\{p^{\left[\dfrac{\log m}{\log 2}+1\right]}:p\in\mathbb{P},p\leq m+1\right\}\:,\]
hence for all $m\geq1$ we have
\[\varphi^{-1}(m)\subseteq\left\{1,\ldots,\prod\left\{p^{\left[\dfrac{\log m}{\log 2}+1\right]}:p\in\mathbb{P},p\leq m+1\right\}\right\}\:.\]
Thus for all $m\geq1$, $\varphi^{-1}(m)$ is finite and
$\varphi$ is finite fibre.
\\
{\bf 3.} For $k\geq2$ we have
${\mathbb P}\subseteq \omega^{-1}\cap\Omega^{-1}(1)\cap d_k^{-1}(k)$
thus $\omega,\Omega,d_k$ are not finite fibre.
\end{note}
\begin{jadval}
By Table~\ref{taha80} and Note~\ref{taha90} we have the following table (where ``$-$'' indicates
that for the corresponding case $\lambda$ is not finite fibre
and contravariant set--theoretical entropy of $\lambda$ is
undefined):
\[\begin{array}{l|c|c|}
 \lambda & {\rm ent}_{\rm set}(\lambda) & {\rm ent}_{\rm cset}(\lambda) \\ \hline
\varphi(=J_1) & 0 & +\infty \\ \hline
\Omega,\omega & 0 & - \\ \hline
J_2, \psi(=\psi_1) & +\infty & 0 \\ \hline
\sigma_k,\psi_k, J_{k+2} (k\geq1) & >0 & 0 \\ \hline
\end{array}\]
\end{jadval}
\begin{problem}
Consider $k\geq1$:
\\
$\bullet$ Compute ${\mathfrak a}(\varphi^*),{\mathfrak o}(d_{k+2}),
{\mathfrak a}(d_{k+2})$.
\\
$\bullet$ For $\lambda=\sigma_k,\psi_{k+1},J_{k+2}$ compute
${\rm ent}_{set}(\lambda)$.
\end{problem}
\section{Some notes on Euler's totient function and Alexandroff topologies on $\mathbb N$}
\noindent We call topological space $X$ Alexandroff, if
intersection of any nonempty family of open sets is open \cite{A}.
In Alexandroff topological space $(X,\tau)$ for every $x
\in X$ we denote the smallest open neighbourhood of $x\in X$ with $V(x,\tau)$. For $f:X\to X$:
\begin{itemize}
\item $\mathcal{B}=\{\bigcup\{f^{-n}(x):n\geq0\}:x\in X\}$
\item $\overline{\mathcal{B}}=\{\{f^n(x):n\geq0\}:x\in X\}$
\end{itemize}
are basis of Alexandroff topologies on $X$. We call topology
generated by $\mathcal{B}$, functional Alexandroff topology on $X$
(with respect to $f$) and denote this topology by $\tau_f$ \cite{AG}.
We call topology
generated by $\overline{\mathcal{B}}$, Alexandroff topology on $X$
with respect to $f$ and denote this topology by $\overline{\tau}_f$ \cite{R}.
For $f:X\to X$ and $x\in X$, we have:
\begin{itemize}
\item $V(x,\tau_f)=\bigcup\{f^{-n}(x):n\geq0\}$,
\item $V(x,\overline{\tau}_f)=\{f^n(x):n\geq0\}$.
\end{itemize}
As it has been mentioned in \cite{AKS}, set--theoretical entropies of $f:X\to X$ interact with cellularities of the above mentioned Alexandroff spaces on $X$. So we devote this section to arising
Alexandroff topologies from some of number theoretical functions. 
\begin{lemma}\label{salam40}
For $f:{\mathbb N}\to\mathbb{N}$ and $k\in\mathbb{N}$ we have:
\begin{itemize}
\item[1.] if for $n\geq1$ we have $f(n)\geq n$, then 
$V(k,\tau_f)\subseteq\{1,\ldots,k\}$;
\item[2.] if for $n\geq1$ we have $f(n)\leq n$, then 
$V(k,\overline{\tau}_f)\subseteq\{1,\ldots,k\}$.
\end{itemize}
\end{lemma}
\begin{proof}
1) Suppose for all $n\geq1$ we have $f(n)\geq n$. For $k\geq1$ suppose
$x\in V(k,\tau_f)$, then there exists $m\geq0$ with $k=f^m(x)\geq x$.
\\
2) Suppose for all $n\geq1$ we have $f(n)\leq n$. For $k\geq1$ suppose
$x\in V(k,\overline{\tau}_f)$, then there exists $m\geq0$ with $x=f^m(k)\leq k$.
\end{proof}
\begin{lemma}\label{salam50}
For $f:{\mathbb N}\to\mathbb{N}$ and $k\in\mathbb{N}$ 
if for $n>1$ we have $f(n)< n$ and $f(1)=1$, then 
$1\in V(k,\overline{\tau}_f)$ and $(\mathbb{N},\overline{\tau}_f)$ is connected. Also $V(1,\tau_f)=\mathbb{N}$ and $(\mathbb{N},\tau_f)$ is
connected too.
\end{lemma}
\begin{proof}
Suppose $m=\min V(k,\overline{\tau}_f)$, then 
$f(m)\in V(k,\overline{\tau}_f)$, so 
\[m\geq f(m)\geq \min
V(k,\overline{\tau}_f)\]
and $m=f(m)$ hence $m=1$ and 
$1\in V(k,\overline{\tau}_f)$. Since $1$ belongs to every
nonempty subset of $(\mathbb{N},\overline{\tau}_f)$, thus it
does not have any disjoint nonempty open subset and it is connected.
\\
For all $n\geq1$ we have $f^n(n)=1$, so $n\in V(1,\tau_f)$ and
$V(1,\tau_f)=\mathbb{N}$. For nonempty open subsets of $U,V$
of $(\mathbb{N},\tau_f)$ with $\mathbb{N}=U\cup V$ we may
suppose $1\in U$ so $V(1,\tau_f)\subseteq U$ and $\mathbb{N}=U$
which leads to connectivity of $(\mathbb{N},\tau_f)$.
\end{proof}
\begin{lemma}\label{salam60}
For $f:{\mathbb N}\to\mathbb{N}$ suppose
for $n>1$ we have $f(n)\geq n$ and $f(1)=1$, then 
$(\mathbb{N},\overline{\tau}_f)$ and $(\mathbb{N},\tau_f)$ are
disconnected.
\end{lemma}
\begin{proof}
$\{1\},\mathbb{N}\setminus\{1\}$ is a separation of 
$(\mathbb{N},\overline{\tau}_f)$ (and $(\mathbb{N},\tau_f)$).
\end{proof}
\begin{example}
For $1\leq\alpha\leq\aleph_0$ suppose $M$ is a partition
of $\mathbb N$ to $\alpha$ infinite subsets of $\mathbb{N}$.
For $D\in M$ suppose $D=\{n_k^D\}_{k\geq1}$ with $n_1^D<n_2^D<\cdots$
and define $f_D:D\to D$ with $f_D(n_k^D)=n_{k+1}^D$ for $k\geq1$, then
for $f:=\mathop{\bigcup}\limits_{D\in M}f_D:\mathbb{N}\to\mathbb{N}$,
we have $f(n)>n$ for all $n\geq1$ and $M$ is the collection of all
connected components of $(\mathbb{N},\tau_f)$.
\end{example}
\begin{jadval}
By Lemmas~\ref{salam50} and~\ref{salam60} we have following table:
\[\begin{array}{l|c|}
\lambda & ({\mathbb N},\tau_\lambda) {\rm \: and \:} ({\mathbb N},\overline{\tau}_\lambda) \\ \hline
\varphi(=J_1),\varphi^*,\omega,\Omega,d(=d_2) & {\rm connected} \\ \hline
\sigma_k,\psi_k,J_{k+1} (k\geq1) & {\rm disconnected}\\ \hline
\end{array}\]
\end{jadval}
\section*{Acknowledgement}
\noindent The authors are grateful to the research division of the University of Tehran
 for the grant which supported this research.

\noindent
{\small  
{\bf Fatemah Ayatollah Zadeh Shirazi},
Faculty of Mathematics, Statistics and Computer Science,
College of Science, University of Tehran,
Enghelab Ave., Tehran, Iran
\linebreak
({\it e-mail}: fatemah@khayam.ut.ac.ir)
\\
{\bf Reza Yaghmaeian},
Faculty of Mathematics, Statistics and Computer Science,
College of Science, University of Tehran,
Enghelab Ave., Tehran, Iran
({\it e-mail}: rezayaghma@yahoo.com)}

\end{document}